\documentclass[12pt]{amsart}
\usepackage{enumitem} 
\usepackage{amssymb}
\usepackage{mathtools}
\usepackage{amsmath}
\usepackage{amsthm} 
\usepackage[utf8]{inputenc}
                      {\hspace*{\fill}\nobreak$\Box$\par\medskip}
                       {\hspace*{\fill}\nobreak$\Box$\par\medskip}
\newtheorem*{conjs}{Conjectures}
\newtheorem{thm}{Theorem}[section]

\newtheorem{remrk}[thm]{Remark}
\newtheorem{defn}[thm]{Definition}
\newtheorem{exmp}[thm]{Example}

\newtheorem{lemma}[thm]{Lemma}

\begin{document}
\title[Pencils of Conic-Line Curves]%
{Pencils of Conic-Line Curves}
\author[H. Suluyer]%
{Hasan~Suluyer}
\email{hsuluyer@metu.edu.tr}
\address{Department of Mathematics, Middle East Technical University, Çankaya, Ankara, 06800 Turkey}
\begin{abstract}
In this paper, we study the restrictions on the number $m$ of conic-line curves in special pencils. The most general result we obtain is the relation between upper bounds on $m$  and the number $p$ of concurrent lines in these pencils. We construct a one-parameter family of pencils such that each pencil in the family contains exactly 4 conic-line curves. We also deal with pencils whose conic-line curves are in general position.
\end{abstract}

\maketitle

\let\thefootnote\relax\footnotetext{ \textbf{Keywords:} pencil of curves, conic-line curves, Euler charateristic, multinets, line arrangements\\

\textbf{2024 Mathematics Subject Classification:} 14N20,14H50,55R55   }

\section{Introduction}
A \textit{conic-line curve} is a union of lines or irreducible conics in the complex projective plane. In \cite{cogolludo2021free}, it was proved that the number of such curves in pencils where all singular points of curves with more than one component have only ordinary multiple points is at most 6. Our main goal is to find the upper bound 6 for the number $m$ of conic-line curves in pencils with transverse intersections at the base points by using the number $p$ of curves which are concurrent lines in such pencils. The number $m$ of curves with certain irreducible decompositions in some pencils of degree $d$ curves plays a special role for the existence of $(m,d)-$multinets. By Noether’s AF+BG theorem, there is an equivalence between the existence of Ceva pencils with $m\geq3$ curves whose irreducible components are only lines and the existence of $(m,d)-$multinets \cite{Falk}. Multinets are used in the study of resonance varieties of complements of complex projective line arrangements. For more details, see \cite{Bartz},\cite{stipins},\cite{milnorfibr}. One can prove upper bounds for $m$, independent of $d$, in order for an $(m,d)$-multinet to exist. It is known that $m$ cannot exceed 4. Also, $m=4$ happens when the multinet is a net. The only known $(4,d)-$net is the Hesse arrangement, which is generated by the curves $C_1: x^3+y^3+z^3=0$ and $C_2: xyz=0$. What if the degree of each irreducible component is at most 2? The set of irreducible components of conic-line curves in a pencil forms a conic-line arrangement. The combinatorics and geometry of conic-line arrangements are studied in \cite{pokoraconicline}. Our main result is the following:\\

\textbf{Main Theorem} (see Theorem \ref{results for general (m,d) for any d})
Let $P=\{ \lambda f(x_0,x_1,x_2)+\mu g(x_0,x_1,x_2)=0 \, | \, [ \lambda:\mu] \in \mathbb{CP}^1 \}$ be a pencil of homogeneous polynomials of degree $d>2$ over $\mathbb{C}$ with transverse intersections at base points and $m$ conic-line curves, $p$ of which are \textit{pencils of $d$ lines}, $d$ concurrent lines. Then,
\begin{enumerate} [label=(\roman*)]
    \item $ m\leq 6$;
    \item if $p=1$, then $ m\leq 5$;
     \item if $p=2$, then $ m\leq 4$;
      \item if $p=3$, then $m=3$.
\end{enumerate}
The paper is organized as follows. In section 2, we recall pencils of conic-line curves and fibrations obtained from such pencils. In section 3, we explore restrictions on the number $m$ of conic-line curves in a pencil of conic-line curves by estimating the Euler characteristic of the associated surface. As a result, for any integer $d>2$, we obtain a relation between upper bounds on $m$ and the number $p$ of curves which are $d$ concurrent lines. In section 4, we construct a one-parameter family of pencils such that each pencil contains exactly 4 conic-line curves. Finally in section 5, we look at pencils of odd degree $d$ curves whose conic-line curves are in general position and obtain that $d\geq 11$. Then, we state conjectures about the non-existence of odd degree curves with 6 conic-line curves in general position and about $m$ being less than 5 for $d>4$.

\subsection*{Acknowledgment} The author would like to thank his advisor Ali Ulaş Özgür Kişisel for his helpful comments and valuable discussions. The author is also thankful to A. Libgober and J.I. Cogolludo for their remarks and  useful discussions.

This paper is a part of the author’s Ph.D. thesis \cite{mythesis} at Middle East Technical University.
\section{Preliminaries}

\subsection{Pencil of conic-line curves}
Let $C_1$ and $C_2$ be degree $d$ curves defined by two homogenous polynomials $f$ and $g$ in $\mathbb{C}[x_0,x_1,x_2]$ respectively without a common factor. Then, the \textit{pencil of curves} through $C_1$ and $C_2$ is defined as $$P=\{ \lambda f(x_0,x_1,x_2)+\mu g(x_0,x_1,x_2)=0 \, | \, [ \lambda:\mu] \in \mathbb{CP}^1 \}.$$ 
Any two distinct curves in this pencil meet at the same set $\mathcal{X}$ of points, which is called the \textit{base locus of the pencil}.
 The number $|\mathcal{X}|$ of distinct points in $\mathcal{X}$ is always less than or equal to $d^2$ by Bézout's theorem. When $C_1$ and $C_2$ intersect transversely, $|\mathcal{X}|=d^2$.\\
 
A curve is called a \textit{conic-line curve} if each of its irreducible components is either a line or a conic. So, a conic-line curve is of the form $ \prod_{i=1}^t \alpha_i^{m_i}$ where $\alpha_i$ are distinct linear or irreducible quadratic forms and $m_i \in \mathbb{Z}_{>0}$ for $1\leq i\leq t$.

\begin{defn}
A pencil $P$ of curves is called a \textit{pencil of conic-line curves} (or conic-line pencil) if the pencil contains at least three distinct conic-line curves. 
\end{defn}
Throughout the paper, assume that $P$ is a pencil of degree $d$ conic-line curves intersecting transversely with $|\mathcal{X}|=d^2$.\\

 Let $m$ be the number of conic-line curves $C_i:\{h_i(x_0,x_1,x_2)=0\}$ in the pencil $P$ for $i \in \{1,\cdots,m\}$. In other words, for $m$ different values of $[ \lambda_i : \mu_i] \in \mathbb{CP}^1$, $h_i=\lambda_i f+\mu_i g$ is the product of irreducible polynomials of the degree 1 or degree 2. Let $\mathcal{A}_i$ be the set of irreducible components of the conic-line curve $C_i$ in the pencil $P$.
 \subsection{Natural fibration over $\mathbb{C P}^1$}
There exists a natural fibration corresponding to a pencil $P$ of conic-line curves with $m$ conic-line curves. By blowing up $\mathbb{C P}^2$ at the points of $\mathcal{X}$, we obtain a surface $S$. Say $\psi: S  \rightarrow \mathbb{CP}^2$ is the blow-up map with  the exceptional curves $E_1, \ldots, E_{d^2}$. Then, there is a fibration
$$
\begin{aligned}
\varphi: S & \rightarrow \mathbb{C P}^1 \\
p \in\{\lambda f+\mu g=0\} & \mapsto [ \lambda : \mu]
\end{aligned}
$$
 This fibered surface has $m$ fibers $W_i$, each of which is the strict transform of a conic-line curve in $P$. Each of these fibers is called a \textit{special fiber}. The set of irreducible components of a special fiber $W_i$ consists of rational curves obtained by the strict transform of lines and conics in $\mathcal{A}_i$. So, each conic-line curve $C_i$ in the pencil $P$ corresponds to a fiber $W_i$ of $\varphi$. Say that the special fibers $W_i$ are over $p_1, p_2, \ldots, p_m \in \mathbb{C P}^1$. There may also be singular fibers of $\varphi$ which are not special. These fibers are called \textit{non-special}. Then, each exceptional curve $E_i$ intersects exactly one rational curve in $\varphi^{-1}\left(p_j\right)$ for any $j$.
 
\section{The Number of Conic-Line Curves in $P$}
In order to bound the number $m$ from above, we estimate the Euler characteristic $e(S)$ of the fibered surface $S$ by calculating the Euler characteristics of smooth and special fibers of $\varphi$ separately. The next lemma provides suitable bounds for the Euler characteristics of special fibers.

\begin{lemma}
\label{characteristics bounds}
   Let $P$ be a pencil of degree $d$ conic-line curves with $|\mathcal{X}|=d^2$.   For any special fiber $W_i$ of $\varphi$, $$\frac{d(5-d)}{2}-q_i \leq e(W_i) \leq d+1 $$ where $q_i$ is the number of irreducible conics in $\mathcal{A}_i$ of the conic-line curve $C_i$ in $P$ corresponding to $W_i$ for each $i$.
\end{lemma}
\begin{proof}
  Say the irreducible components of the conic-line curve $C_i$ corresponding to $W_i$ contains $q_i$ irreducible conics. Then, the remaining irreducible components are $d-2q_i$ lines. As Euler characteristics of strict transforms of lines and conics in $\mathbb{CP}^2$ is 2,
  \begin{eqnarray*}
  e(W_i) &=& 2(d-2q_i)+2q_i- \displaystyle\sum_{p \in W_i} (r_p-1)\\
  &=& 2d-2q_i- \displaystyle\sum_{p \in W_i} (r_p-1) 
  \end{eqnarray*}
  where $r_p$ is the number of strict transforms of lines and irreducible conics through $p$ in the pencil and $r_p\geq1$ as $p\in W_i$. To find the minimum value of $e(W_i)$, we need to find the maximum value of $\displaystyle\sum_{p \in W_i} (r_p-1)$ where $r_p-1 \geq 0$. Each point $p\in W_i$ with $r_p=r>2$ is an intersection point of strict transforms of $r$ lines and smooth conics, and the contribution of the point $p\in W_i$ to this sum is $r-1$. If there were transverse intersections among strict transforms of $r$ lines and smooth conics, there would be $r\choose 2$ points with multiplicities 2 and $$r-1 < {r\choose 2}(2-1)=(r-1)\frac{r}{2}.$$ 
  So, $$\displaystyle\sum_{p \in W_i} (r_p-1)< \displaystyle\sum_{p \in W_i} {r_p\choose 2}.$$
  If each $p\in W_i$ with $r_p=r>1$ is an intersection point of strict transforms of $r$ lines and irreducible conics without common tangent at $p$, the upper bound $\displaystyle\sum_{p \in W_i} {r_p\choose 2}$ becomes $|I|$ where $I$ is the set of intersection points of $d-2q_i$ lines and $q_i$ irreducible conics such that they intersect transversely at $p$ for any intersection point $p$, which are called \textit{lines and conics in general position}. If there is a common tangent at some of the intersection point $p\in W_i$, $$\displaystyle\sum_{p \in W_i} {r_p\choose 2} < |I|.$$

  Then, $$\displaystyle\sum_{p \in W_i} (r_p-1)\leq \displaystyle\sum_{p \in W_i} {r_p\choose 2}\leq |I|. $$  The set $I$ consists of points each of which is a transverse intersection of two lines, two irreducible conics or one line and one irreducible conic. So, for $W_i$ in general position,
  \begin{eqnarray*}     
      |I|&=& {d-2q_i \choose 2}+4{q_i \choose 2}+2(d-2q_i)q_i\\
           &=& \frac{d^2-d-2q_i}{2}={d \choose 2}-q_i.\\ 
  \end{eqnarray*}
  Then, the number of nodes in the special fiber $W_i$ is
  \begin{equation}
  \label{nodes in special fibers}
      {d \choose 2}-q_i.
  \end{equation}
   Therefore, $$e(W_i) = 2d-2q_i-|I|= \dfrac{d(5-d)}{2}-q_i \text{ at most}.$$ Now, we try to find the maximum value of $e(W_i)$ by minimizing $\displaystyle\sum_{p \in W_i} (r_p-1)$. \\
   If $q_i=0$, the irreducible components of the conic-line curve $C_i$ consists of $d$ distinct lines and so these lines have to be a \textit{pencil of lines}, the set of lines meeting at a single point, to minimize $\displaystyle\sum_{p \in W_i} (r_p-1)$." So,
   $$e(W_i)=2d-2q_i-\displaystyle\sum_{p \in W_i} (r_p-1) = 2d-(d-1)=d+1 .$$ If $q_i \neq 0$, 
   \begin{eqnarray*}
       e(W_i)&=&2d-2q_i-\displaystyle\sum_{p \in W_i} (r_p-1)\\
       &=&2d-2q_i-\left(\displaystyle\sum_{p \in I_1} (r_p-1)+\displaystyle\sum_{p \in I_2} (r_p-1)\right)
   \end{eqnarray*}
   where $I_1$ is the set of intersection points where at least one conic passes and $I_2$ is the set of intersection points where only lines meet among the irreducible components of $C_i$. Then,
   $$e(W_i)<2d-2q_i-\displaystyle\sum_{p \in I_2} (r_p-1)\leq 2d-2q_i -(d-2q_i-1) = d+1 $$
 since $\displaystyle\sum_{p \in I_1} (r_p-1)>0$ for $q_i \neq 0$ and $\displaystyle\sum_{p \in I_2} (r_p-1)$ is minimum when $d-2q_i$ lines form a pencil. Hence, $$e(W_i)\leq d+1 .$$ 
 \end{proof}
The following lemma characterizes when the maximum value $d+1$ for the Euler characteristic of a special fiber is attained.
 
\begin{lemma}\label{special case e(W_i)=d+1}
Let $P$ be a pencil of degree $d$ conic-line curves with $|\mathcal{X}|=d^2$. For any special fiber $W_i$ of $\varphi$, $e(W_i)=d+1$  if and only if $W_i$ is the strict transform of a pencil of $d$ lines, i.e. $d$ lines meeting at a single point.
\end{lemma}
\begin{proof}
  It is clear that if $W_i$ is the strict transform of a pencil of $d$ lines, then these $d$ lines meet at a single point outside $\mathcal{X}$. Then, $W_i$ is a pencil of $d$ rational curves and $e(W_i)=2d-(d-1)=d+1$. Conversely, assume that $e(W_i)=d+1$. Suppose that the irreducible components of the conic-line curve $C_i$ corresponding to $W_i$ contains $q_i$ irreducible conics. Then, the other irreducible components are $d-2q_i$ lines. First, we want to exhibit that $q_i=0$ when $e(W_i)=d+1$. Assume that $q_i \neq 0$. Then, 
   \begin{eqnarray*}
       e(W_i)&=&2d-2q_i-\displaystyle\sum_{p \in W_i} (r_p-1)\\
       &=&2d-2q_i-\left(\displaystyle\sum_{p \in I_1} (r_p-1)+\displaystyle\sum_{p \in I_2} (r_p-1)\right)
   \end{eqnarray*}
   where $I_1$ is the set of intersection points where at least one conic passes and $I_2$ is the set of intersection points where only lines meet among the irreducible components of $C_i$. By Lemma \ref{characteristics bounds},
   $$e(W_i)< 2d-2q_i-\displaystyle\sum_{p \in I_2} (r_p-1) \leq 2d-2q_i-(d-2q_i-1)= d+1 $$
 since $\displaystyle\sum_{p \in I_1} (r_p-1)>0$ as $q_i \neq 0$. Due to our assumption that $e(W_i)=d+1$, we deduce that $q_i = 0$. Then, the irreducible components of the conic-line curve $C_i$ consists of $d$ distinct lines. To obtain $e(W_i)=d+1$, they have to be a pencil of lines.
 \end{proof}
 
We get an upper bound on $m$ similar to the upper bound obtained in Theorem 3.2 in \cite{Yuz-1} by estimating the Euler characteristic of $S$, by making use of the above bounds on the Euler characteristics of special fibers.

\begin{thm} \label{results for general (m,d) for any d}
Let $P=\{ \lambda f(x_0,x_1,x_2)+\mu g(x_0,x_1,x_2)=0 \, | \, [ \lambda:\mu] \in \mathbb{CP}^1 \}$ be a pencil of homogeneous polynomials of degree $d>2$ over $\mathbb{C}$ with transverse intersections at base points and $m$ conic-line curves, $p$ of which are \textit{pencils of $d$ lines}, i.e. $d$ concurrent lines. Then,
\begin{enumerate} [label=(\roman*)]
    \item $ m\leq 6$;
    \item if $p=1$, then $ m\leq 5$;
     \item if $p=2$, then $ m\leq 4$;
      \item if $p=3$, then $m=3$.
\end{enumerate}
    
\end{thm}
\begin{proof}
First of all, $0 \leq p \leq m$. By Lemma \ref{special case e(W_i)=d+1}, the Euler characteristic of each special fiber corresponding to a pencil of $d$ lines is $d+1$. Without loss of generality, $e(W_i)=d+1$ for $i=1,2,...,p.$ Let us estimate $e(S)=3+d^2$ by using Euler characteristics of special fibers of $S$. As the generic fiber is a smooth degree $d$ curve, its Euler characteristics is always less than that of the other singular fibers and its Euler characteristic is $2-2g$ where $$g=\dfrac{(d-1)(d-2)}{2}.$$ So, it is equal to $3d-d^2$. Then, we obtain
$$3+d^2\geq (2-m)(3d-d^2)+p(d+1)+(m-p)\dfrac{d(5-d)}{2}-\displaystyle\sum_{i=p+1}^m q_i$$ by using Lemma \ref{characteristics bounds}. Say $\bar q=\displaystyle\sum_{i=p+1}^m q_i$. Then,

\begin{eqnarray*}  
    3+d^2 &\geq& (2-m)(3d-d^2)+p(d+1)+(m-p)\dfrac{d(5-d)}{2}-\bar q \\ 
    &\geq& m\dfrac{d(d-1)}{2}+p\dfrac{(d-1)(d-2)}{2}-\bar q+6d-2d^2.
\end{eqnarray*}
So, $$ m\dfrac{d(d-1)}{2}\leq 3(d-1)^2-p\dfrac{(d-1)(d-2)}{2}+\bar q.$$
As $d>1$,
\begin{equation}
\label{ineq: m}
    m \leq  6\dfrac{d-1}{d}-p\dfrac{d-2}{d}+\bar q\dfrac{2}{d(d-1)}   
\end{equation}
Now, let us find an upper bound for $\bar q$. As it is clear that $q_i\leq \left\lfloor\dfrac{d}{2}\right\rfloor \text{ for any } i$,  $$\bar q \leq (m-p) \left\lfloor\dfrac{d}{2}\right\rfloor.$$
First, assume that $d$ is an odd integer. Then, $$\bar q\leq (m-p)\dfrac{d-1}{2} \text{ for any } i.$$ Then, inequality \eqref{ineq: m} becomes
\begin{eqnarray*}
      m &\leq&  6\dfrac{d-1}{d}-p\dfrac{d-2}{d}+\dfrac{m-p}{d}=\dfrac{d-1}{d}(6-p)+\dfrac{m}{d}. 
\end{eqnarray*} 
We deduce that 
\begin{equation}
\label{ineq: m for odd d}
    m \leq  6-p.
\end{equation}
As $0\leq p\leq m \leq  6-p$, $p$ is less than or equal to 3. So, $p \leq m\leq 6 \text{ for any odd } d.$  By using inequality \eqref{ineq: m for odd d}, we can conclude the last three statements. \\

When $d$ is even, $\bar q\leq (m-p)\dfrac{d}{2} \text{ for any } i$ and $d\geq 4$. Then, inequality \eqref{ineq: m} becomes  \begin{eqnarray*}
      m &\leq&  6\dfrac{d-1}{d}-p\frac{d-2}{d}+\dfrac{m-p}{d-1}=6\dfrac{d-1}{d}-p\dfrac{d^2-2d+2}{d(d-1)}+\dfrac{m}{d-1}. 
\end{eqnarray*}
So, \begin{eqnarray*}
m&\leq& 6\dfrac{(d-1)^2}{d(d-2)}-p\dfrac{d^2-2d+2}{d(d-2)}=6\left(1+\dfrac{1}{d(d-2)}\right)-p\left(1+\dfrac{2}{d(d-2)}\right)\\
&<& \dfrac{27}{4}-p
\end{eqnarray*}
since $0<\frac{1}{d(d-2)}\leq  \frac{1}{8}$ for $d\geq 4$. Since $m$ and $p$ are positive integers,
$$m<\dfrac{27}{4}-p \implies m\leq 6-p.$$ Therefore, the theorem holds for even integer values of $d$ as well.

\end{proof}
The following pencil is an example  for the case of $p =3$ and $d = 3$ in Theorem \ref{results for general (m,d) for any d} and it consists of exactly $m=3$ conic-line curves.
\begin{exmp} Consider the Fermat pencil 
    $$P=\{ \lambda (x^d-y^d)+\mu (y^d-z^d)=0 \, | \, [ \lambda:\mu] \in \mathbb{CP}^1 \}$$ for $d>2$. Then $|\mathcal{X}|=d^2$. The special fibers of $P$ corresponding to the conic-lines  $x^d-y^d=0, y^d-z^d=0$ and $x^d-z^d=0$  have Euler characteristics $d+1$ as the irreducible components of each of these conic-lines are $d$ distinct lines through a single point outside the base locus. So, $p \geq 3$. By Theorem \ref{results for general (m,d) for any d}, $p=m=3$. Therefore, the Fermat pencil has no conic-line except the curves $x^d-y^d=0,\,y^d-z^d=0$ and $x^d-z^d=0$. In fact, each fiber of $S$ except the fibers corresponding these special fibers is smooth due to the equality $$e(S)=3+d^2=(2-3)(3d-d^2)+3\big(2d-(d-1)\big).$$
\end{exmp}
\section{A Family of Pencils of Conic-Line Curves with $m=4$}
The Hesse pencil is the only known pencil with four curves whose irreducible components are lines only up to projective isomorphism. It is conjectured to be the unique $(4,d)$-net, for instance see \cite{stipinphd},\cite{gunturkunpd}. On the other hand, there are infinitely many pencils with four conic-line curves. We have constructed the following family of pencils in the case of $p=0,d=3$ and $m=4$.
\begin{thm} \label{pencils P_a for (4,3) } For any $a\in \mathbb{C}-\{0,1\} $, the pencil $P_a=\{ \lambda (x-ay)(x^2+y^2-z^2)+\mu (x-y)(x^2-xy+y^2-z^2)=0 \, | \, [ \lambda:\mu] \in \mathbb{CP}^1 \}$  over $\mathbb{C}$ has exactly $4$ conic-line curves whose equations are: 
\begin{align*} 
    C_1&: (x-ay)(x^2+y^2-z^2)=0,\\
    C_2&: (x-y)(x^2-xy+y^2-z^2)=0,\\
    C_3&:y ((2-a) x^2 -xy+(1-a) y^2 + (a - 1)z^2)=0,\\
    C_4&:x ((1-a) x^2 + a x y +(1- 2 a) y^2 + (a- 1)z^2)=0.
\end{align*}

\end{thm}
\begin{proof}
    First of all, $C_1,C_2,C_3$ and $C_4$ are conic-line curves and they are found in the pencil when the $[ \lambda:\mu]$ values are $[1:0],[0:1],[1:-1]$ and $[1:-a]$ respectively. Then, we can find $4$ distinct lines each of which is an irreducible component of different conic-line curves. These lines intersect at the point $[0:0:1]$ in the base locus $\mathcal{X}$. As each of these lines consists of 3 points from $\mathcal{X}$ and the point $[0:0:1] \in \mathcal{X}$ is the intersection point of all these 4 lines, the other 2 points of each of line in $\mathcal{X}$ are different from the other lines. Therefore, the number of points on these 4 lines in $\mathcal{X}$ is $$1+2+2+2+2=9=|\mathcal{X}|.$$ 
    If the given pencil had one more conic-line curve, it would have an irreducible component $l$ with degree 1 due to odd degree. This $l$ would consist of 3 points from $\mathcal{X}$. Since there is at least one line among the first 4 lines through any point in the base locus and $[0:0:1]$, $l$ cannot pass through $[0:0:1]$. If $l$ does not pass through $[0:0:1]$, it would have three points from $\mathcal{X}$, each of which is on three of four different lines. Then, the remaining line would also be an irreducible component of a different conic-line curve  which would contradict that there exists no common component in the curves of the pencil $P$. 
   So, there are exactly 4 conic-line curves. Moreover, if $a \neq \frac{1}{2} $ or $2$ as well, all above conic-line curves are in general position. For $a=\frac{1}{2}$, $C_4$ only and for $a=2$, $C_3$ only are not in general position. 
\end{proof}
These pencils have singular curves different from these four conic-line curves. As an example of this, we find all singular curves in the pencil $P_2$. These are 4 conic-line curves and 3 nodal curves.
\begin{exmp}
      Consider the pencil $$P_2=\{ \lambda (x-2y)(x^2+y^2-z^2)+\mu (x-y)(x^2-xy+y^2-z^2)=0 \, | \, [ \lambda:\mu] \in \mathbb{CP}^1 \}.$$ 
     The special curves of this pencil are 
     \begin{align*}
         C_1&: (x-2y)(x^2+y^2-z^2)=0, \\ C_2&: (x-y)(x^2-xy+y^2-z^2)=0, \\ C_3&: y(-xy- y^2 + z^2)=0, \\ 
         C_4&: x(-x^2+2xy-3y^2+z^2)=0 
     \end{align*} when  $[\lambda:\mu]$ values are  $[1:0],[0:1],[1:-1]$, and $[1:-2]$ respectively as in Figure \ref{pencil of curves with m=4 }. Also, this pencil has 3 non-special singular fibers, each of which  has a node as a singularity. The irreducible components of each of $C_1,C_2$ and $C_4$ are one line and one irreducible conic meeting at two distinct points while those of $C_3$ are one line and one irreducible conic meeting at a single point. So, the Euler characteristics of the fibers corresponding to conic-line curves $C_1,C_2,C_3$ and $C_4$ are $2,2,3$ and $2$, respectively. The Euler characteristics of each of the non-special singular fibers is 1. As the Euler characteristics of generic fiber is $3d-d^2=0$ for the degree $d=3$, the following equality holds:
    $$e(S)=3+3^2=12= (2-6)\cdot 0 +\underbrace{2+3+2+2}_{\text{Euler Char. of special fibers }} + 3 \cdot 1 .$$ 
\end{exmp}
 \begin{figure}[hbt!]
  
	\centering                                                     \includegraphics[width=0.5\linewidth]{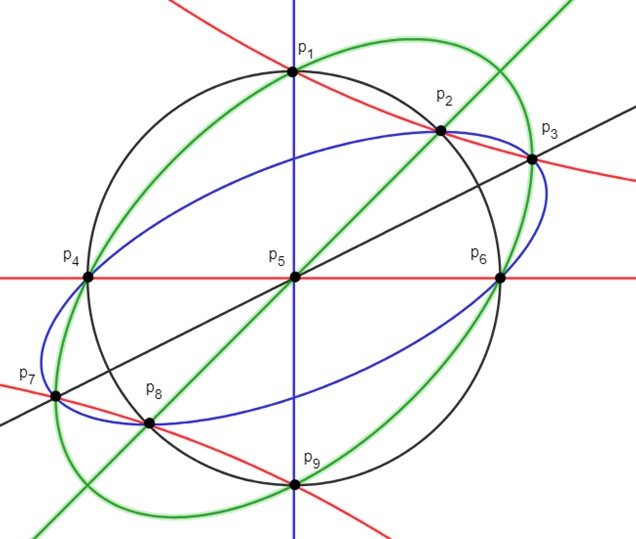}   
\caption{The conic-line curves of the pencil $P_2$}
\label{pencil of curves with m=4 }
\end{figure}
\section{Conic-Line Curves in General Position}
Our focus in this section shifts to the conic-line pencils whose conic-line curves are in general position. 
\begin{defn}
    A conic-line curve is in general position if its irreducible components intersect transversely and no three of them are concurrent.
\end{defn}
Our first observation in this section is on the contribution of the singular points in the non-special fibers of $S$ to its Euler characteristic.

\begin{lemma}
\label{contr. of non-special fibers to characteristic}
Let $P$ be a pencil of conic-line curves of odd degree $d$ over $\mathbb{C}$ with $m$ conic-line curves. If each of the conic-line curves is the union of lines and conics in general position, then the contribution of the singular points in the non-special fibers of the fibered surface $S$ to $e(S)$ is at most $\frac{(d-1)^2(6-m)}{2}$.
    
\end{lemma}
\begin{proof}
     By using formula \eqref{nodes in special fibers}, we calculate the number $\delta _s$ of nodes in special fibers as  $$m{d \choose 2}-\displaystyle\sum_{i=1}^m q_i.$$ Then, the contribution of singular points in the non-special fibers of $S$ to the Euler characteristic of $S$ is
 $$ e(S)-e(\mathbb{C P}^1)\cdot e(\Sigma_g)-\delta _s$$ where $g=\frac{(d-1)(d-2)}{2}$, as each node increases the Euler characteristics by 1. So, this contribution is 
 \begin{eqnarray*}
    3+d^2-2(-d^2+3d)-\delta _s &=&3(d-1)^2-m{d \choose 2}+\displaystyle\sum_{i=1}^m q_i\\
     &=&3(d-1)^2-m\frac{d(d-1)}{2}+ \displaystyle\sum_{i=1}^m q_i.
    \end{eqnarray*}
    Since $d$ is odd, we can use the upper bound $\frac{d-1}{2}$ for each $q_i$. Then, 
\begin{eqnarray*}
     3(d-1)^2-m\frac{d(d-1)}{2}+ \displaystyle\sum_{i=1}^m q_i &\leq& 3(d-1)^2-m\frac{d(d-1)}{2}+m\frac{d-1}{2}\\
     &\leq& \frac{(d-1)^2}{2}\big(6-m\big).     
\end{eqnarray*}
\end{proof}
\begin{remrk} When $d$ is even, $q_i \leq \frac{d}{2}$ and the contribution of the singular points in the non-special fibers to $e(S)$ is at most $3(d-1)^2-m\frac{d(d-2)}{2}$ which is bigger than the bound in Lemma \ref{contr. of non-special fibers to characteristic}. Also, this bound becomes 3 if $m=6$.
\end{remrk}
The next observation is about the distribution of conics and lines in the set of irreducible components of such conic-line curves in a hypothetical odd degree conic-line pencil with $m=6$.
\begin{thm}
    \label{possibilities for (6,d) in general pos}  Let $P$ be a pencil of conic-line curves of odd degree $d$ over $\mathbb{C}$ with $m=6$ conic-line curves. If each of the conic-line curves is a union of lines and conics in general position, then  
\begin{enumerate} [label=(\roman*)]
    \item the singular curves of the pencil are only these conic-line curves,
    \item each conic-line curve in the pencil is a union of 1 line and $\frac{d-1}{2}$  irreducible conics.
    \item there exists at least 1 conic passing through each point in $\mathcal{X} $ as an irreducible component of a conic-line curve if $d=3$.
\end{enumerate} 

\end{thm}
 \begin{proof} 
   By above Lemma \ref{contr. of non-special fibers to characteristic}, the contribution of singular points in the non-special fibers of $S$ to Euler characteristics of $S$ is at most 0 when $m=6$. As this contribution is always non-negative as well, this is zero. It means that each non-special fiber is smooth. Therefore, the singular curves of the pencil are only these conic-line curves.\\
   
 By the previous statement, there exists no singular fiber of $\varphi$ except the special fibers. So, $$ e(S)-e(\mathbb{C P}^1)\cdot e(\Sigma_g)=3(d-1)^2$$ is the number $\delta _s$ of nodes in the special fibers of $S$. By formula \eqref{nodes in special fibers}, 
  \begin{eqnarray*}
       \delta_s=3(d-1)^2=6{d \choose 2}-\displaystyle\sum_{i=1}^6 q_i \implies \displaystyle\sum_{i=1}^6 q_i &=&3d(d-1)-3(d-1)^2\\
        &=&3(d-1).
  \end{eqnarray*}
   Moreover, we  know that the irreducible components of each conic-line curve contains at most  $\frac{d-1}{2}$ conics and so $$\displaystyle\sum_{i=1}^6 q_i\leq 3(d-1).$$ 
   So, $\displaystyle\sum_{i=1}^6 q_i$ takes its maximum value $3(d-1)$ whenever the irreducible components of each conic-line curve contains the largest number of conics. Therefore, each of the conic-line curves is a union of 1 line and $\frac{d-1}{2}$ conics in general position. \\
   
     Assume to the contrary that there exists a point $p$ in $\mathcal{X}$ such that the lines each of which is an irreducible component of a different conic-line curve meet at $p$ and $d=3$. Then, the number of base points on these lines is at least $6(d-1)+1=6d-5$. So,
  $$6d-5\leq d^2 \implies 0\leq(d-1)(d-5).$$ Then, $d\geq 5$, a contradiction.
  
 \end{proof}
 If $m<6$,  the contribution of the singular points in the non-special fibers of the fibered surface $S$ may not be zero. 
 
  Ruppert’s example in \cite{cogolludo2021free} is an example for the case $m=6$ and $d=3$. We cannot find a pencil $P$ of odd degree $d>3$ curves with $6$ conic-line curves, each of which is in general position. In fact, we cannot find a pencil $P$ of degree $d> 4$ curves with $m=5$ or 6. We make the following conjectures.
  \begin{conjs}
   There does not exist a pencil $P$ of odd degree $d>3$ curves with $6$ conic-line curves, each of which is in general position. The number $m$ has to be less than or equal to 4 for a pencil of conic-line curves for large values of $d$.
  \end{conjs}
\bibliographystyle{ieeetr} 
\bibliography{Pencils.bib}

\begin{thebibliography}{10}

\bibitem{cogolludo2021free}
J.~I. Cogolludo and A.~Libgober, ``Free quotients of fundamental groups of smooth quasi-projective varieties,'' {\em Proceedings of the Edinburgh Mathematical Society}, vol.~64, no.~4, pp.~924--946, 2021.

\bibitem{Falk}
M.~Falk and S.~Yuzvinsky, ``Multinets, resonance varieties, and pencils of plane curves,'' {\em Compositio Mathematica}, vol.~143, no.~4, pp.~1069--1088, 2007.

\bibitem{Bartz}
J.~Bartz, ``Induced and complete multinets,'' in {\em Configuration Spaces: Geometry, Topology and Representation Theory}, pp.~213--231, Springer, 2016.

\bibitem{stipins}
J.~Stipins, ``Old and new examples of $k$-nets in $\mathbb{P}^2$,'' {\em arXiv preprint math/0701046}, 2007.

\bibitem{milnorfibr}
G.~Denham and A.~I. Suciu, ``Multinets, parallel connections, and {Milnor} fibrations of arrangements,'' {\em Proceedings of the London Mathematical Society}, vol.~108, no.~6, pp.~1435--1470, 2014.

\bibitem{pokoraconicline}
P.~Pokora and T.~Szemberg, ``Conic-line arrangements in the complex projective plane,'' {\em Discrete \& Computational Geometry}, vol.~69, no.~4, pp.~1121--1138, 2023.

\bibitem{mythesis}
H.~Suluyer, {\em Classification Problems of Multinets and Conic-Line Arrangements}.
\newblock PhD thesis, Middle East Technical University, 2024.

\bibitem{Yuz-1}
S.~Yuzvinsky, ``Realization of finite abelian groups by nets in $\mathbb{P}^2$,'' {\em Compositio Mathematica}, vol.~140, no.~6, pp.~1614--1624, 2004.

\bibitem{stipinphd}
J.~Stipins~III, {\em On finite k-nets in the complex projective plane}.
\newblock PhD thesis, 2007.

\bibitem{gunturkunpd}
M.~H. G{\"u}nt{\"u}rk{\"u}n, {\em Using tropical degenerations for proving the nonexistence of certain nets}.
\newblock PhD thesis, 2010.

\end{thebibliography}
\end{document}